\documentclass[reqno]{amsart}
\usepackage{setspace}
\usepackage{amsmath}
\usepackage{dsfont}
\usepackage{amsfonts}
\usepackage{amssymb}
\usepackage{aliascnt}
\usepackage{graphicx}
\usepackage{mathrsfs}
\usepackage{mathtools}
\usepackage{cases}
\usepackage{enumitem}
\setenumerate[1]{label={\!(\arabic*)}}
\usepackage[bookmarks=true,pdfstartview=FitH, pdfborder={0 0 0}, colorlinks=true,citecolor=blue, linkcolor=blue]{hyperref}
\usepackage{tikz}
\usepackage[nameinlink]{cleveref}
\usepackage{prettyref}
\hypersetup{pdfborder = {0 0 0}}

\everymath{\displaystyle}
\newtheorem{theorem}{Theorem}
\numberwithin{theorem}{section}

\newaliascnt{conj}{theorem}
\newaliascnt{corollary}{theorem}
\newaliascnt{lemma}{theorem}
\newaliascnt{fact}{theorem}
\newaliascnt{claim}{theorem}
\newaliascnt{prop}{theorem}
\newaliascnt{definition}{theorem}
\newaliascnt{remark}{theorem}

\newtheorem{thmx}{Theorem}
\numberwithin{thmx}{section}

\newtheorem{conjx}{Conjecture}

\newtheorem{corollary}[corollary]{Corollary}
\newtheorem{lemma}[lemma]{Lemma}

\newtheorem{prop}[prop]{Proposition}

\newtheorem{remark}[remark]{Remark}

\aliascntresetthe{claim}
\aliascntresetthe{conj}
\aliascntresetthe{corollary}
\aliascntresetthe{definition}
\aliascntresetthe{fact}
\aliascntresetthe{lemma}
\aliascntresetthe{prop}
\aliascntresetthe{remark}

\newtheoremstyle{claimintheorem}
{\topsep}   
{\topsep}   
{}          
{1.5em}          
{\bf}          
{.}          
{3pt plus 1pt minus 1pt} 
{}          
\theoremstyle{claimintheorem}
\newcounter{claimintheorem}
\newtheoremstyle{claiminproof}
{\topsep}   
{\topsep}   
{}          
{1.5em}          
{\bf}          
{.}          
{3pt plus 1pt minus 1pt} 
{}          
\theoremstyle{claiminproof}
\newcounter{claiminproof}
\newenvironment{proofof}[1]{{\noindent\textit{Proof of #1.}}}{\hfill $\Box$\par}

\newenvironment{sketchofproof}{{\noindent\textbf{Sketch of proof.}}}{\hfill $\Box$\par}

\crefname{claim}{Claim}{Claim}
\crefname{corollary}{Corollary}{Corollary}
\crefname{conj}{Conjecture}{Conjecture}
\crefname{definition}{Definition}{Definition}
\crefname{equation}{Equation}{Eq.}
\crefname{example}{Example}{Example}
\crefname{lemma}{Lemma}{Lemma}
\crefname{prop}{Proposition}{Proposition}
\crefname{remark}{Remark}{Remark}
\crefname{claimintheorem}{Claim}{Claim}
\crefname{claiminproof}{Claim}{Claim}
\crefname{theorem}{Theorem}{Theorem}

\theoremstyle{definition}
\newaliascnt{example}{theorem}

\aliascntresetthe{example}

\def\sek~{\S{}}

\numberwithin{equation}{section}

\makeatletter

\newcommand{\Rmnum}[1]{\expandafter\@slowromancap\romannumeral #1@}
\makeatother

\newcommand{\dvol}{\operatorname{dVol}}

\newcommand{\Ric}{\operatorname{Ric}}

\newcommand{\trace}{\operatorname{tr}}

\newcommand{\cF}{{\mathcal F}}

\newcommand{\cH}{{\mathcal H}}

\newcommand{\cJ}{{\mathcal J}}

\newcommand{\cR}{{\mathcal R}}

\begin{document}
\title[Extensions of the Bonnet-Myers Theorem]{Extensions of the Bonnet-Myers Theorem}

\author{Ronggang Li}
\address[Li]{School of Mathematical and statistics, Nanjing University of Information Science and Technology. Nanjing 210044, China}
\email{lrg@pku.edu.cn}

\author{Shaoqing Wang}
\address[Wang]{School of Mathematics and Statistics, Key Laboratory of Nonlinear Analysis and Applications (Ministry of Education), and Hubei Key Lab--Math. Sci., Central China Normal University, Wuhan 430079, China}
\email{swang@ccnu.edu.cn}


\begin{abstract}
In this paper, we present extensions of the classical Bonnet-Myers theorem for Riemannian manifolds with nonnegative Ricci curvature. Our results provide criteria for compactness and a method for estimating the diameter of such manifolds under general curvature conditions. As applications, we establish compactness theorems for manifolds whose Ricci curvature decays at polynomial or exponential rates.
\end{abstract}
\maketitle


\section{Introduction}\label{section:introduction}

Geometric inequalities involving the radial Ricci curvature along a geodesic are key to determining the compactness of a manifold and to estimating its diameter if the manifold is closed. This can be confirmed by the classical Bonnet-Myers theorem and its generalizations.

In \cite{calabi1967}, Calabi gives a series of integral estimates of the square root of the average radial Ricci curvature $q(r) = \frac{1}{n-1} \Ric(\dot\gamma, \dot\gamma)$ along a minimal geodesic $\gamma$ on a Riemannian manifold with non-negative Ricci curvature. In particular, when $\gamma: [0, \infty) \to M$ is a ray,
\begin{equation*}\label{ineq:cal}
\int_{a}^{b} \sqrt{q(r)}  dr \leq \sqrt{\frac{1}{4} \ln^2\left(\frac{b}{a}\right) + \ln\left(\frac{b}{a}\right)}, \quad \forall b > a > 0.
\end{equation*}
Therefore, if $(M, g)$ is complete and non-compact (i.e., open), the decay rate of its Ricci curvature is at most $2$. Equivalently, $(M, g)$ is compact provided $q(r) \geq c r^{-\lambda}$ for $r \geq r_0$, where $c > 0$, $\lambda < 2$, and $r$ is the distance to a fixed point $p \in M$.

One can verify that if $\lambda = 2$ and $c = \frac{1}{4} + \mu > \frac{1}{4}$, then $M$ is also compact; furthermore, Cheeger, Gromov, and Taylor proved that its diameter is at most $\frac{e^\mu}{r_0}$ \cite{CGT1982}.

If $\lambda> 2$, Wan showed that $M$ is compact provided $c \geq \frac{(p-1)^p}{(p-2)^{p-2}} r_0^{p-2}$ \cite{wan2019}.

In this paper, general lower bounds on Ricci curvature are considered, and we establish technical refinements of the Bonnet-Myers theorem as follows. 

\begin{theorem}\label{thm:technical}
Let $(M, g)$ be a complete $n$-dimensional Riemannian manifold. Suppose $\gamma: [0, \infty) \to M$ is a ray along which the average radial Ricci curvature satisfies $q(r) = \frac{1}{n-1} \Ric(\dot\gamma, \dot\gamma)(r) \geq 0$. Let $\psi$ be a monotonic function on $[0, \infty)$. Define
\begin{align}\label{def:functional}
\cF(q, \psi, a, b) =
\begin{dcases}
\int_{a}^{b} \frac{\psi q}{\psi^2 + q}  dr & \text{if $\psi$ is non-decreasing on $[a, b]$},  \\
\int_{a}^{b} \frac{\psi q}{\psi^2 + q}  dr - \frac{1}{4} \log\frac{\psi(a)}{\psi(b)} & \text{if $\psi$ is non-increasing on $[a, b]$}, \\
\end{dcases}
\end{align}
where $[a, b] \subset [0, \infty)$ and $b$ may be taken as $\infty$. Then,
\begin{align}\label{ineq:criterion}
\cF(q, \psi, a, b) \leq 1,
\end{align}
\end{theorem}

\begin{remark}
 \Cref{thm:technical} can also deduce that
\begin{align}\label{ineq:mean-curvature}
m(a) \geq \psi(a) \frac{\cF(q, \psi, a, b)}{1 - \cF(q, \psi, a, b)}.
\end{align}
where the meaning of the function $m$ can be seen in \Cref{eq:m=v'/v}.
\end{remark}

\begin{corollary}\label{cor:compact-crit}
If \Cref{ineq:criterion} does not hold for each $\gamma(r) = \exp_p(r \nu)$, $r \in [0, \infty)$, then $M$ is compact.
\end{corollary}

For the case where $\gamma$ is a geodesic segment on $M$, we obtain the following theorem:

\begin{theorem}\label{thm:technical2}
Let $(M, g)$ be a complete $n$-dimensional Riemannian manifold. Suppose $\gamma: [0, \rho] \to M$ is a geodesic such that $\gamma(\rho)$ is conjugate to $\gamma(0)$ along $\gamma$, and the average radial Ricci curvature $q(r)$ along $\gamma$ is non-negative. Let $\psi$ be a positive and monotonic function on $[0, \infty)$. Then
\begin{align}\label{ineq:segment}
\int_{0}^{\rho} \frac{\psi q}{\psi^2 + q}  dr < 2 + \frac{1}{4} \left| \log\frac{\psi(\rho)}{\psi(0)} \right|.
\end{align}
\end{theorem}

Based on this result, the following improvement of the Bonnet-Myers theorem can be deduced directly:

\begin{theorem}\label{thm:gbm}
Let $(M, g)$ be a complete $n$-dimensional Riemannian manifold with a base point $p$. Suppose $r$ is the distance to $p$, and $$q(r) = \inf\left\{ \Ric(\nu, \nu) : \nu \in S_x M,~r(x) = r \right\}$$ for all $r  \geq 0$. If there exists a constant $l > 0$ and a positive monotonic function $\psi$ on $[0, \infty)$ such that
\begin{align}\label{criterion:segment}
\int_{0}^{l} \frac{\psi q}{\psi^2 + q}  dr \geq 2 + \frac{1}{4} \left| \log\frac{\psi(l)}{\psi(0)} \right|,
\end{align}
then the diameter of $M$ is at most $l$.
\end{theorem}

\begin{remark}
One can verify $\frac{\partial}{\partial q}\frac{\psi q}{\psi^2+q}\geq 0$ if $\psi>0$.
\end{remark}

From the above results, we can also obtain a consequence as follows.

\begin{theorem}
Let $(M, g)$ be a complete $n$-dimensional Riemannian manifold. Suppose $\gamma_{\nu}(r) = \exp_p(r\nu)$, $r \in [0, \infty)$ is a ray with initial direction $\nu \in S_pM$, along which the average radial Ricci curvature satisfies $q(r) = \frac{1}{n-1} \Ric(\dot\gamma, \dot\gamma)(r) \geq 0$. Then
\begin{align*}
\int_0^{\infty} \frac{q}{1+q} dr \leq 1.
\end{align*}
Moreover, $M$ is compact if there exists a point $p \in M$ such that the inequality above fails for every geodesic $\gamma_{\nu}$ emanating from $p$.

If $r$ denotes the distance to $p$, and $q(r) = \inf\left\{ \Ric(\nu, \nu) : \nu \in S_x M,~r(x) = r \right\}$, then the diameter of $M$ is at most $l$ provided
\begin{align*}
\int_0^{l} \frac{q}{1+q}(r)dr \geq 2.
\end{align*}
\end{theorem}

The rest part of this paper is organized as follows. \Cref{section:preliminaries} provides some necessary preliminaries that support the proof of the main Theorems. In \Cref{ssection:proofs}, we give the proofs of Theorem 1.3, 1.4, 1.5. \Cref{section:appl}  is devoted to some applications of our main results.

\section{Preliminaries and Basic Results}\label{section:preliminaries}

Let $(M, g)$ be a complete (closed or open) Riemannian manifold without boundary. Suppose $p \in M$, and let $\nu \in S_p M$, i.e., $\nu \in T_p M$ such that $|\nu| = 1$. Let $\gamma_{\nu}(r) = \exp_p(r \nu)$, where $r \in [0, \infty)$ is an arc-length parameterized geodesic on $M$; we denote it briefly by $\gamma: [0, \infty) \to M$. Let $\{e_i\}_{i=0}^{n-1}$ be an orthonormal frame at $p = \gamma(0)$ such that $e_0 = \dot\gamma(0)$. The normal Jacobi fields $\{J_i\}_{i=1}^{n-1}$ along $\gamma$ satisfy $J_i(\gamma(0)) = 0$ and $\nabla_{\dot\gamma(0)} J_i = e_i$ for $i = 1, \dots, n-1$.

Along the geodesic $\gamma$, we denote the representation matrix of $(J_1, \cdots, J_{n-1})$ with respect to the basis $(e_1, \cdots, e_{n-1})$ at $\gamma(r)$ by $\cJ(r)$, and the matrix of the linear map $\nu \mapsto R(\nu, \dot\gamma)\dot\gamma$ by $\cR$. The Jacobi equation yields

\begin{align}\label{eq:Jacobi}
\cJ'' + \cR \cJ = 0.
\end{align}

Before the conjugate point, $\cJ$ is invertible. Let $\cH = \cJ' \cJ^{-1}$; then we obtain the following matrix-valued Riccati equation:

\begin{align}\label{eq:Riccati}
\cH' + \cH^2 + \cR = 0.
\end{align}

We define the mean volume density $v(r)$ by
\begin{align}\label{def:mean-volume}
v^{n-1}(r) = \dvol(\dot\gamma, J_1, \dots, J_{n-1}) = \det \cJ,
\end{align}
so that
\[
v(0) = 0, \qquad v'(0) = 1.
\]

Recall Jacobi's formula for determinants:
\[
(\det \cJ)' = \det \cJ \cdot \trace(\cH).
\]

We define the radial mean curvature along $\gamma$ by
\begin{align}\label{eq:m=v'/v}
m(r)
= \frac{1}{n-1} \trace(\cH)
= \frac{1}{n-1} \frac{(v^{n-1})'(r)}{v^{n-1}(r)}
= \frac{v'}{v}(r).
\end{align}

The average radial Ricci curvature is defined by
\[
q(r) = \frac{1}{n-1} \Ric(\dot\gamma, \dot\gamma)(r).
\]

The Riccati inequality involving the radial Ricci curvature along $\gamma$ is obtained by taking the trace of \Cref{eq:Riccati} and applying the Cauchy inequality. Using our notations above, we can rewrite it as:
\begin{equation}\label{ineq:Riccati}
m' + m^2 + q \leq 0.
\end{equation}

If $q(r) \geq 0$ for all $r \geq 0$, then $m$ is non-increasing. We define $\zeta = \sup \{ r : m(r) > 0 \}$ to be the mean convex radial radius; note that it depends on the direction $\nu$, i.e., $\zeta = \zeta(\nu)$.

If $q(r) \geq 0$ for all $r \geq 0$ and $m(a) < 0$ for some $a > \zeta$, then $m'(r) \leq -m^2(r) < 0$ implies $m(r) \leq \frac{1}{r - (a - 1/m(a))}$, which forces $m$ to go to $-\infty$ before $a - \frac{1}{m(a)}$, i.e., $m$ blows down to $-\infty$ in finite time. In this case, $v$ also goes to $0$ at $\gamma(\rho)$. We define $\rho$ to be the radial conjugate radius in the direction $\nu$. The radial conjugate radius may be $\infty$, but if $\rho$ is finite, there exists $l \leq \rho$ such that $\gamma$ is a geodesic segment on $[0, l]$. As a corollary, if $\gamma$ is a ray, there is no conjugate point to $p$ along $\gamma$; therefore, $m(r) \geq 0$ for all $r \geq 0$. Furthermore, if $q > 0$ along $\gamma$, then $m > 0$ for all $r \geq 0$ and $\zeta = \infty$.

Under the condition that the Ricci curvature of $M$ is non-negative, the integral of the radial Ricci curvature $q(r)$ or its functional along a geodesic segment $\gamma: [0, l] \to M$ is limited; that is, it must satisfy certain geometric inequalities. For example, the second variation formula for geodesics implies that for $[a, b] \subset [0, l]$ and a smooth function $u$ on $[a, b]$,
\begin{align}\label{ineq:the-second-variational-formula}
\int_a^b u^2(r) q(r)  dr \leq u^2(b) m(b) - u^2(a) m(a) + \int_a^b u'^2(r)  dr.
\end{align}

The classical Bonnet-Myers theorem can therefore be deduced from Wirtinger's inequality for smooth functions $u$ on $[a, b]$ such that $u(a) = u(b) = 0$:
\[
\int_a^b u^2(r)  dr \leq \frac{(b-a)^2}{\pi^2} \int_a^b u'^2(r)  dr.
\]

Furthermore, if $\gamma: [0, \infty) \to M$ is a ray and $[a, b] \subsetneq [0, \infty)$, then \Cref{ineq:Riccati} and the fact that $m(b) > 0$ imply
\[
\int_a^b q(r)  dr \leq m(a).
\]
Since $m(a) \leq \frac{1}{a}$ by the mean curvature comparison, we obtain the following geometric inequality by letting $b \to \infty$:
\begin{align}\label{ineq:1/t-upper-bound}
\int_a^\infty q(r)  dr \leq \frac{1}{a}, \quad \forall a > 0.
\end{align}

In the following, we introduce a technical lemma from ODE cf. $\S 9, \textrm{\uppercase\expandafter{\romannumeral13}}$, \cite{odeww}:

\begin{lemma}\label{lem:exist-squeeze}
Let $\cF(r,y)$ be a continuous function on $[0,\infty)\times(-\infty,+\infty)$, and it is Lipschitz continuous with respect to $y$ (for instance, suppose $\dfrac{\partial \cF}{\partial y}$ is a continuous function). $y_1(r)$ and $y_2(r)$ are differentiable functions in $[0,l]$ satisfying the inequalities
\begin{equation*}
y_1\leq y_2;
\end{equation*}
\begin{equation*}
{y_2}'-\cF(r,y_2)\leq 0 \leq {y_1}'-\cF(r,y_1).
\end{equation*}
Then the differential equation
\begin{equation*}
y'=\cF(r,y)
\end{equation*}
has a global solution $\phi$ in $[0,l]$ with $y_1\leq \phi \leq y_2$, and the conclusion still holds if we replace the finite interval $[0,l]$ by $[0,\infty)$.
\end{lemma}

\begin{sketchofproof}
In the case that the interval is finite i.e. $[0,l]$, let $\phi$ be a left-moving solution of $y'=\cF(r,y)$ with the right-hand side initial-value $\phi(l)=y_2(l)$, then $y_1\leq \phi \leq y_2$ on $[0,l]$ can be implied by the left-moving comparison of ODE.

If the interval is $[0,\infty)$, let $\phi_n$ be the solution of $y'=\cF(r,y)$ on $[a,n]$ such that $\phi_n(n)=y_2(n)$. For any fixed interval $[0,l]$, the sequence of functions $\{\phi_n\}$ decreases monotonically, they are also uniformly bounded from above by $y_2$ and bounded from below by $y_1$. Therefore, $\{\phi_n\}$ uniformly converges to a function $\phi$ on $[0,l]$, which is also a solution to $y'=\cF(r,y)$. For the arbitrariness of $l\in[0,\infty)$, we have $\phi$ on $[0,\infty)$ globally.
\end{sketchofproof}


\section{Proof of Main Theorems}\label{ssection:proofs}

The notations in this section are same as the ones in \Cref{section:preliminaries}.
\begin{prop}\label{prop:range-of-psiphi}
Let $\gamma:[0,\infty)\rightarrow M$ be a ray on $M$ along which the average radial Ricci curvature $q(r)\geq 0$. Then, for any positive and derivable function $\psi$ on $[a,\infty)$, there exists a function $\phi$ on $[a,\infty)$ such that:
\begin{enumerate}
    \item
    \begin{equation}\label{eq:differential-equation-of-phi}
    \phi'+\psi'\phi^2-\frac{q}{\psi^2+q}=0,
    \end{equation}
    \item
    \begin{equation}\label{ineq:range-of-psiphi}
    \frac{\psi}{\psi+m}\leq \psi\phi \leq 1.
    \end{equation}
\end{enumerate}
\end{prop}

\begin{proof}
For the fixed functions $q(r)$ and $\psi(r)$, we define
\begin{align*}
\cF(r,y)=\frac{q(r)}{q(r)+\psi^2(r)}-\psi'(r)y^2.
\end{align*}
A direct calculation implies
\begin{align*}
\left(\frac{1}{\psi}\right)'(r)-\cF\left(r,\dfrac{1}{\psi}\right)=-\dfrac{q(r)}{q(r)+\psi^2(r)}\leq 0
\end{align*}
on $[a,\infty)$.

On the other hand, the Riccati inequality \eqref{ineq:Riccati} implies:
\begin{align*}
-m'\left(\psi^2+q\right)\geq \left(m^2+q\right)(\psi^2+q)\geq q\left(\psi+m\right)^2.
\end{align*}
Since $\psi$ is assumed to be positive here and $m\geq 0$ since $\gamma$ is a ray,
\begin{align*}
\frac{-m'}{\left(\psi+m\right)^2} \geq \frac{q}{\psi^2+q},
\end{align*}
by differentiating $\left(\frac{1}{\psi+m}\right)$ and combining the inequality above, we have
\begin{align}\label{ineq:first}
\big(\frac{1}{\psi+m}\big)'+\psi'\big(\frac{1}{\psi+m}\big)^2-\frac{q}{\psi^2+q}\geq 0.
\end{align}
which implies
$$(\frac{1}{\psi+m})'(r)-\cF(r,\frac{1}{\psi+m})\geq 0.$$
Therefore, by applying \Cref{lem:exist-squeeze}, there exists $\phi$ on $[0,\infty)$ such that:
\begin{align*}
\phi'+\psi'\phi^2-\frac{q}{\psi^2+q}=0,\\
\frac{1}{\psi+m}\leq \phi \leq \frac{1}{\psi}.
\end{align*}
The proof is finished.
\end{proof}

Now we are in a position to prove our main theorems.

\begin{proofof}{\Cref{thm:technical}}
We see \eqref{eq:differential-equation-of-phi} implies:
\begin{align}\label{eq:differetial-equation-of-psiphi}
(\psi\phi)'=(\log \psi)' \psi\phi(1-\psi\phi)+\frac{\psi q}{\psi^2+q}.
\end{align}
If $\psi$ is non-decreasing, \eqref{eq:differetial-equation-of-psiphi} implies:
\begin{align*}
(\psi\phi)'\geq  \frac{\psi q}{\psi^2+q}.
\end{align*}
If $\psi$ is a decreasing function, (\ref{eq:differetial-equation-of-psiphi}) implies:
\begin{align*}
(\psi\phi)'\geq \frac{1}{4}(\log \psi)'+  \frac{\psi q}{\psi^2+q}.
\end{align*}
By integrating both sides of the differential inequalities above, and using \eqref{ineq:range-of-psiphi}, we have
\begin{align}
\frac{m(a)}{\psi(a)+m(a)}\geq
\begin{dcases}
\int_{a}^{b}\frac{\psi q}{\psi^2+q} dr &\text{if $\psi$ is non-decreasing on $[a,b]$},  \\
\int_{a}^{b}\frac{\psi q}{\psi^2+q} dr-\frac{1}{4}\log\frac{\psi(a)}{\psi(b)} &\text{if $\psi$ is non-increasing on $[a,b]$}. \\
\end{dcases}
\end{align}
Therefore \eqref{ineq:criterion} and \eqref{ineq:mean-curvature} follows directly.
\end{proofof}


\begin{proof}[Proof of \Cref{thm:technical2}]
If $\gamma:[a,b]\rightarrow M$ is a geodesic segment such that $m(b)\geq 0$, the conclusions in \Cref{thm:technical} hold. Let $a=0$, $b=\zeta$ the radial mean convex radius, then

\begin{align}\label{ineq:0tozeta}
\int_{0}^{\zeta}\frac{\psi q}{\psi^2+q} dr \leq
\begin{dcases}
1 &\text{if $\psi$ is non-decreasing on $[0,\zeta]$},  \\
1+\frac{1}{4}\log\frac{\psi(0)}{\psi(\zeta)} &\text{if $\psi$ is non-increasing on $[0,\zeta]$}. \\
\end{dcases}
\end{align}

Since $m(r)\leq 0$, $\forall r\in[\zeta,\rho]$,
$$
\frac{m^2+q}{(m-\psi)^2}\geq\frac{q}{\psi^2+q}
$$
for any positive function $\psi$ on $[\zeta,\rho]$, since $m'+m^2+q\leq 0$,
\begin{align*}
\big(\frac{1}{\psi-m}\big)'+\psi'\big(\frac{1}{\psi-m}\big)^2+\frac{q}{\psi^2+q}\leq 0.
\end{align*}
A direct calculus implies
\begin{align*}
0\leq \left(\frac{1}{\psi}\right)'(r)+\cF\left(r,\dfrac{1}{\psi}\right)+\frac{q}{\psi^2+q},
\end{align*}
since
\begin{align*}
\frac{1}{\psi-m}\leq \frac{1}{\psi}
\end{align*}
on $[\zeta,\rho]$, there exists another $\phi$ on $[\zeta,\rho]$ such that
\begin{align*}
\phi'+\psi'\phi^2+\frac{q}{\psi^2+q}=0,\\
\frac{1}{\psi-m}\leq \phi \leq \frac{1}{\psi}.
\end{align*}
Then
\begin{align*}
(\psi\phi)'=(\log \psi)' \psi\phi(1-\psi\phi)-\frac{\psi q}{\psi^2+q}.
\end{align*}
Since $0\leq \psi\phi\leq 1$ on $[\zeta,\rho]$,
\begin{align}\label{ineq:zetatorho}
\int_{\zeta}^{\rho} \frac{\psi q}{\psi^2+q}dr\leq
 \begin{dcases*}
     1+\frac{1}{4}\log \frac{\psi(\rho)}{\psi(\zeta)}  &\text{if $\psi$ is non-decreasing on $[\zeta,\rho]$},  \\
     1                   &\text{if $\psi$ is non-increasing on $[\zeta,\rho]$}. \\
 \end{dcases*}
\end{align}
Then we obtain the following inequality by combining \eqref{ineq:0tozeta} and \eqref{ineq:zetatorho},
\begin{align*}
\int_{0}^{\rho} \frac{\psi q}{\psi^2 + q}  dr < 2 + \frac{1}{4} \left| \log\frac{\psi(\rho)}{\psi(0)} \right|,
\end{align*}
where $\psi$ is positive and monotonic on $[0,\rho]$.
\end{proof}

\section{Applications}\label{section:appl}
\begin{corollary}\label{thm:criterion-poly-decay}
Let $M$ be a Riemannian manifold with Ricci curvature bounded below by a rational function of the form $(n-1)q(r)=(n-1)\frac{c}{r^p}$ for $r\geq a>0$, where $r$ is the distance to a fixed point $o\in M$. If $p>2$, then $M$ is compact provided
\begin{align}\label{ineq:criterion-poly-decay}
c>c(p,a):=
\frac{p^{2p} \sin^p(\frac{\pi}{p})}{4\pi^p}\left(\frac{a}{p-2}\right)^{p-2}.
\end{align}
\end{corollary}
\begin{remark}
One can verify that this constant is sharper than the one in \cite{wan2019}.
\end{remark}

\begin{proof}
Let $\psi$ be a constant $x>0$. Then
\begin{align*}
\cF\left(\frac{c}{r^p},x,a,\infty\right)&=\int_a^{\infty} \dfrac{\dfrac{cx}{r^p}}{\dfrac{c}{r^p}+x^2}dr
     =\int_{\left(\frac{x^{2}}{c}\right)^{1/p}a}^{\infty}\frac{c^{\frac{1}{p}}x^{1-\frac{2}{p}}}{1+\tau^p}  d\tau\\
     &=c^{\frac{1}{p}}x^{1-\frac{2}{p}}\left(\frac{\pi/p}{\sin(\pi/p)}-
     \int_{0}^{\left(\frac{x^{2}}{c}\right)^{1/p}a}\frac{1}{1+\tau^p}  d\tau\right)\\
     &>c^{\frac{1}{p}}x^{1-\frac{2}{p}}\left(\frac{\pi/p}{\sin(\pi/p)}-\left(\frac{x^2}{c}\right)^{1/p}a\right)\\
     &=-a x +\frac{\pi/p}{\sin(\pi/p)}c^{\frac{1}{p}}x^{1-\frac{2}{p}}.
\end{align*}
A direct calculation shows that the maximum of the expression on the right-hand side over $x>0$ is
    \begin{align*}
    \max_{x>0}\left\{-a x +\frac{\pi}{p\sin(\pi/p)}c^{\frac{1}{p}}x^{1-\frac{2}{p}}\right\}
    =2c^{\frac{1}{2}}\left(\frac{\pi}{\sin(\pi/p)}\right)^{\frac{p}{2}}\left(\frac{p-2}{a}\right)^{\frac{p-2}{2}}\frac{1}{p^p},
    \end{align*}
attained at $x= \left(\dfrac{c^{\frac{1}{p}}\pi(p-2)}{ap^2\sin(\pi/p)}\right)^{\frac{p}{2}}$.
Therefore, \Cref{cor:compact-crit} implies $M$ is compact if
    $$
    c>\frac{p^{2p} \sin^p(\pi/p)}{4\pi^p}\left(\frac{a}{p-2}\right)^{p-2}.
    $$
\end{proof}
\begin{remark}
If $p=2$,
\begin{align*}
\cF\left(\frac{c}{r^2},\frac{\sqrt{c}}{r},a,b\right)=\left({\sqrt{c}}-\frac{1}{4}\right)\log\frac{b}{a}.
\end{align*}
Therefore, $M$ is compact if $c>\dfrac{1}{4}$.
If $p<2$,
\begin{align*}
\cF\left(\frac{c}{r^p},\sqrt{\frac{c}{r^{p}}},a,b\right)=\sqrt{c}\left(b^{1-\frac{p}{2}}-a^{1-\frac{p}{2}}\right)-\frac{p}{8}\log\frac{b}{a},
\end{align*}
which can be arbitrarily large. Therefore $M$ is always compact for any $c>0$.
\end{remark}

\begin{corollary}
Let $M$ be a Riemannian manifold with Ricci curvature bounded below by an exponential-type function of the form $q(r) = c e^{-p r}$, where $r$ is the distance to $o\in M$ and $p$ is a positive constant. If $ c\geq \frac{2p^2}{\log^2 3}$, then $M$ is compact; and if $c > (e^2 - 1)p^2$, the diameter of $M$ is at most $\frac{1}{p} \log\frac{e^2 c}{c - (e^2 - 1)p^2}$.
\end{corollary}

\begin{proof}
Let us choose a test function and compute the associated functional. Consider
\begin{align*}
\int_t^{\infty} \frac{x c e^{-p r}}{x^2 + c e^{-p r}}  dr &= \frac{x}{p} \log\left(1 + \frac{c e^{-p t}}{x^2}\right).
\end{align*}
We seek parameters $x, t \geq 0$ such that this integral equals a specific value (like the bound from our theorem), i.e., we want to solve
\begin{align*}
\frac{x}{p} \log\left(1 + \frac{c e^{-p t}}{x^2}\right) = 1.
\end{align*}
Rearranging gives the condition:
\begin{align*}
(x^2 t + x) \log\left(1 + \frac{c e^{-p t}}{x^2}\right) = p.
\end{align*}
The left-hand side of this equation does not achieve its maximum at $x=\sqrt{\dfrac{c}{2}}, t=0$. Substituting these values yields $c = \frac{2p^2}{\log^2 3}$. Therefore, by \Cref{cor:compact-crit}, $M$ is compact if $c \geq \frac{2p^2}{\log^2 3}$, providing the stated criterion for compactness.

On the other hand, for the diameter estimate, we compute the integral over a finite interval:
\begin{align*}
\int_0^{\rho} \frac{x c e^{-p r}}{x^2 + c e^{-p r}}  dr &= \frac{x}{p} \log\left(\frac{x^2 + c}{x^2 + c e^{-p \rho}}\right).
\end{align*}
Let $x = p$. If we also have $c = (e^2 - 1)p^2$, one can verify that the right-hand side of the equality above equals $2$ precisely when $\rho = \frac{1}{p} \log\frac{e^2 c}{c - (e^2 - 1)p^2}$. Therefore, by \Cref{thm:gbm}, the diameter of $M$ is at most this value $\rho$ if $c > (e^2 - 1)p^2$.
\end{proof}

\subsection*{Acknowledgement}
R. Li is supported by the Open Project Program of Key Laboratory of Mathematics and Complex System (Grant No.202501), Beijing Normal University. S. Wang is supported by NSFC No. 12401203 and 12361032, the Guangdong Basic and Applied Basic Research Foundation No. 2023A1515010599, the Self-determined Research Funds of CCNU from the Colleges Basic Research and Operation of MOE  No. CCNU24XJ008, the Fundamental Research Funds for the CCNU (No. CCNU25JCPT032) and the Fundamental Research Funds for the Central Universities No. CCNU25JC026.

\bibliographystyle{alpha}
\bibliography{ref}
\end{document}